\documentclass[11pt,draft]{article}

\usepackage{amsfonts}
\usepackage{amsmath}
\usepackage{amsthm}
\usepackage{amssymb}

\newcommand{\beq}{\begin{equation}}
\newcommand{\eeq}{\end{equation}}
\newcommand{\bea}{\begin{eqnarray}}
\newcommand{\eea}{\end{eqnarray}}
\newcommand{\beas}{\begin{eqnarray*}}
\newcommand{\eeas}{\end{eqnarray*}}

\newtheorem{theorem}{Theorem}[section]

\newtheorem{definition}[theorem]{Definition}
\newtheorem{proposition}[theorem]{Proposition}

\newtheorem{corollary}[theorem]{Corollary}
\newtheorem{lemma}[theorem]{Lemma}

\newtheorem{remark}[theorem]{Remark}
\newtheorem{example}[theorem]{Example}
\newtheorem{examples}[theorem]{Examples}
\newtheorem{foo}[theorem]{Remarks}
\newtheorem*{acknowledgements}{Acknowledgements}

\newcommand{\E}{\mathbb E}

\newcommand{\bM}{\mathbb M}

\newcommand{\M}{\mathbb M}

\newcommand{\R}{\mathbb R}

\newcommand{\supp}{\operatorname{supp}}

\newcommand{\Dom}{\operatorname{Dom}}

\title{ Poincar\'e inequality and the uniqueness of solutions for the heat equation associated with subelliptic
diffusion operators }
\author{ Bumsik Kim }

\date{Department of Mathematics, Purdue University \\
 West Lafayette, IN, USA}

\begin{document}
\maketitle

\begin{abstract}
In this paper we study global Poincar\'e inequalities on balls in a large class of sub-Riemannian manifolds satisfying
the generalized curvature dimension inequality introduced by F.Baudoin and N.Garofalo.
As a corollary, we prove the uniqueness of solutions for the subelliptic heat equation.
Our results apply in particular to CR Sasakian manifolds with Tanaka-Webster-Ricci curvature bounded from below and Carnot groups of step two.
\end{abstract}

\tableofcontents

\section{Introduction} \label{sec:framework}
\sloppy

\
In this paper, $\M$ is a $C^\infty$ connected finite dimensional manifold endowed with a smooth measure $\mu$. Let $L$ be a second-order diffusion operator on $\M$, which is symmetric, non-positive, locally subelliptic in the sense of \cite{FP1}, \cite{JSC}, \cite{Fo}, with $L1=0$.

\
By the subellipticy of $L$, there is an intrinsic distance $d(x,y)$ associated to $L$ on $\M$ (\cite{FP1},\cite{JSC}).
If we denote $\Gamma(f):=\Gamma(f,f)$ (the Bakry-\'Emery's \textit{carr\'e du champ} of $L$, see \cite{BE}) by the quadratic differential form
$ \Gamma(f,g) =\frac{1}{2}(L(fg)-fLg-gLf),$  $f,g \in C^\infty(\M),$
the distance $d(x,y)$ is defined via the notion of subunit curves with respect to the \emph{length of the gradient}, $\sqrt{\Gamma(f)}$.
Throughout this paper, we assume that our sub-Riemannian metric space $(M,d)$ is complete.

\
This class of operators includes H\"ormander type operators of order $k$ (or sum of squares of vector fields satisfying H\"ormander condition), Grushin's operator and
sub-Laplacians on Sasakian CR-manifolds (see \cite{JSC}, \cite{BG}).
For domains whose diameters are bounded in terms of the subellipticity constants, the Poincar\'e inequality was proved in \cite{Je} for the H\"ormander type operator, and
the general subelliptic Poincar\'e inequality can be found in Kusuoka and Stroock's work \cite{KS} through their probabilistic methods.

\
In general, the global Poincar\'e inequality could fail if the domain is not bounded (\cite{Je}, even in the Riemannian case).
On the Riemannian manifold with Ricci curvature bounded from below by $-K<0$, the global Poincar\'e inequality is proved by Buser \cite{Bu}:
\begin{align} \label{ineq:Buser}
 \int_{B(x,r)} | f- f_{B(x,r)} |^2 d\mu \leq C_1 r^2 e^{C_2 \sqrt{K} r } \int_{B(x,r)} | \nabla f|^2 d\mu, \quad \forall f \in C^\infty(\M), r>0.
\end{align}
This was done through geodesic arguments equipped with geometric tools such as the Bishop-Gromov comparison theorem.

\
In our subelliptic framework, due to the lack of several properties such as ellipticity of $L$ or smoothness of the distance function,
we cannot use many of the tools for Riemannian manifolds. As a recent breakthrough, \textit{the generalized curvature dimension inequality} $CD(\rho_1,\rho_2,\kappa,d)$
of Baudoin and Garofalo \cite{BG} was introduced to specify a certain subelliptic curvature condition (\cite{BG},\cite{BB},\cite{BBG},\cite{BBGM},\cite{BG_Riesz},\cite{BauWa},\cite{BauKim}). (Definition \ref{GCDI} in the present paper)

\
For example, in \cite{BG} it was shown that if the lower bound of Tanaka-Webster-Ricci tensor on a CR Sasakian manifold of codimension $1$ is given by $\rho_1 \in \R$,
the sub-Laplacian of $\M$ satisfies $CD(\rho_1,d/4,1,d)$ where the real dimension of the distribution of the CR manifold is $d$.

\
If $\rho_1$ is non-negative (an analogue of the non-negative lower bound of Ricci tensor in the Riemannian manifold),
some types of Poincar\'e inequalities were discussed in \cite{BBG}, \cite{BauKim}, including Buser's Poincar\'e inequality with $K=0$.

\
Our purpose in the present paper is to prove the Buser's global Poincar\'e inequality in the subelliptic framework with the condition $CD(-K,\rho_2,\kappa,d)$, $K> 0$:
\begin{theorem} [Poincar\'e inequality]\label{thm:buser subelliptic}
 If $\M$ satisfies $CD(-K,\rho_2,\kappa,d)$ with $K>0$, for any $r>0$, $x_0\in \M$ and $f\in C^\infty(\M)$,
\begin{align} \label{ineq:Buser subelliptic}
 \int_{B(x_0,r)} |f(x)-f_{B(x_0,r)} |^2 d\mu(x) \leq C_{p1} r^2 e^{C_{p2} Kr^2} \int_{B(x_0,r)} \Gamma(f) d\mu ,
\end{align}
where $C_{p1},C_{p2} >0 $ depend on $\rho_2,\kappa,d$.
\end{theorem}

\
Once the global subelliptic Poincar\'e inequalities and the volume doubling property are assumed, one can obtain Sobolev inequalities through \cite{SC2}, and we can run
Moser's iteration (\cite{Mo1},\cite{Mo2}) to recover various properties such as the uniqueness of the solution for the heat equation
through the mean value estimates of subsolution of $L$-Laplace equation.
Note that two sided heat kernel bounds and parabolic Harnack inequality for $P_t f$ were already obtained through Li-Yau type inequalities and heat kernel methods in \cite{BG},\cite{BBGM}. However this Harnack inequality does not directly yield the uniqueness of positive solutions in contrast to Li-Yau's \cite{LY}.

\
We consider the following heat equation associated with $L$.
\begin{align}\label{heat_eq}
  &\left(
         \frac{\partial}{\partial t} -L \right) u(x,t) =0
\end{align}

\
Adapting the ideas of \cite{DON},\cite{Li}, we conclude our main corollaries, the uniqueness of the solutions for the subelliptic heat equation.

\begin{theorem}[Uniqueness of positive solutions] \label{thm:uniq of positive}
Assume that $\M$ satisfies the \emph{generalized curvature-dimension inequality} $CD(-K,\rho_2,\kappa,d)$, $K>0$ with respect to $L$.
For any non-negative solution $u \in C(\M\times [0,\infty))$ of (\ref{heat_eq}), $u$ is uniquely determined by the initial data $u(\cdot,0) = f$ and
$u(x,t)\equiv P_t f(x)$, where $P_t$ is the heat semigroup generated by $L$.
\end{theorem}

\begin{theorem} [Uniqueness of $L^p$ solutions] \label{thm:uniq of Lp} For $1< p <\infty$, $L^p$ solution $u(x,t)$ of (\ref{heat_eq}) defined on $\M\times (0,\infty)$ is
uniquely determined by the initial data $f\in L^p(\M)$, i.e. $u(\cdot,t) \in L^p(\M) $ for any $t>0$ and $u\xrightarrow{L^p} f$ as $t\rightarrow 0$. \\
If $CD(-K,\rho_2,\kappa,d)$, $K>0$ is satisfied, we have an unique
$L^1$ solution $u(x,t)$ of (\ref{heat_eq}) with the initial data $u\xrightarrow{L^1} f\in L^1(\M)$ as $t\rightarrow 0$.
\end{theorem}

\
Theorem \ref{thm:uniq of positive} improves Li-Yau/Harnack inequality for the subelliptic $L$ which is proved for $P_t f$ in \cite{BG},\cite{BBGM}.
(Find notations in the following section. See Remark \ref{rmk:A epsilon} for $\mathcal A_\epsilon$.)
\begin{corollary} [Li-Yau type inequality, \cite{BG}] \label{Cor:LiYau}
Assume $CD(\rho_1,\rho_2,\kappa,d)$, $\rho_1\in \R$. Any non-negative solution $u(x,t)\in \mathcal A_\epsilon$ of (\ref{heat_eq})
satisfies the following inequality:
\begin{align*}
 &\Gamma( \ln u ) + \frac{2\rho_2}{3} t \Gamma^Z(\ln u ) -\left( 1+\frac{3\kappa}{2\rho_2}-\frac{2\rho_1}{3}t \right) \frac{Lu}{u}\\
  &\leq \frac{d\rho_1^2}{6}t - \frac{d\rho_1}{2}\left(1+\frac{3\kappa}{2\rho_2}\right)
      + \frac{d}{2t} \left(1+\frac{3\kappa}{2\rho_2} \right)^2 .
\end{align*}
\end{corollary}

\begin{corollary} [Harnack inequality, \cite{BBGM},\cite{BG}] \label{Cor:harnack}
Assume $CD(-K,\rho_2,\kappa,d)$, $K\geq 0$. Then any non-negative solution $u \not\equiv 0$ of (\ref{heat_eq}) satisfies for $x,y\in \M$, $s<t\in \R_+$,
\begin{align*}
 \frac{u(x,s)}{u(y,t)} \leq \left(\frac{t}{s}\right)^{\frac{D}{2}} \exp\left(\frac{d K }{4}(t-s) + \frac{K}{12} d(x,y)^2 +
  \frac{d(x,y)^2}{4(t-s)}\left( 1+\frac{3\kappa}{2\rho_2} \right) \right).
\end{align*}
\end{corollary}

\section{Preliminaries }

\
Our subelliptic operator $L$ is described as follows:
$L$ is a second-order diffusion operator with real $C^\infty$ coefficients on $\M$.
There exists a neighborhood $U$ of $x\in\M$ and a constant $C>0$, such that for any $f\in C_0^\infty(U)$,
\begin{align} \label{subellipticity}
 \| f \|_\epsilon ^2 \leq C(\left| \left< f , Lf \right> \right| + \| f\|_2^2 ) ,
\end{align}
where $\|f \|_\epsilon=\left(\int |\hat{u}(\xi)|^2 (1+| \xi|^2 )^\epsilon d\xi \right)^{1/2}$ is the Sobolev norm of order $0<\epsilon<1$,
and $\left<\cdot,\cdot\right>, \|\cdot\|_2 $ are respectively the inner product and the norm of $ L^2 (\M,\mu ) $. Also $L$ is symmetric, non-positive and has
zero order term, i.e.:
\begin{equation*}
 \int_\bM f L g d\mu=\int_\bM g Lf d\mu, \quad \int_\bM f L f d\mu \le 0, \quad L 1 =0,
\end{equation*}
for every $f , g \in C^ \infty_0(\bM)$.

\
The intrinsic sub-Riemannian metric associated with $L$ is defined by the minimal length of subunit curve:
\begin{align*}
 d(x,y) = \inf &\big\{ T \big| \exists \text{ Lipschitz } \gamma:[0,T] \rightarrow \M , \gamma(0)=x,\gamma(T)=y, \\
   &|\frac{d}{dt}f(\gamma(t))|\leq \sqrt{\Gamma f(\gamma(t))}, \forall f\in C^\infty(\M),\text{ almost every }t\in[0,T]  \big\},
\end{align*}
where $ \Gamma(f,g) =\frac{1}{2}(L(fg)-fLg-gLf)$, $\Gamma(f)=\Gamma(f,f)$. We assume that the metric space $(\M,d)$ is complete.

\
Note that this class of operators strictly includes H\"ormander's type operators $-\sum_{i=1}^{m} X_i^* X_i $ ( $X_i$'s are the smooth vector fields satisfying
H\"ormander condition of order $k$ on $\M$, and $X_i^*$ is the formal adjoint of $X_i$ in $L^2(\M,\mu)$ ). Following Strichartz \cite{Str},
the completeness assumption of $(\M,d)$ yields that $L$ is essentially self-adjoint on $C_0^\infty(\M)$. So we can denote by $L$ the unique self-adjoint extension
(the Friedrichs extension) of $L$ in $L^2(\M,\mu)$. Maximum principle(\cite{Bon}) and H\"ormander's hypoellipticity of $L$ are well-known.
See \cite{FP1},\cite{Je},\cite{BG},\cite{NSW} for more properties of $L$.


\
We follow the steps in \cite{BG} to introduce the curvature assumption on our subelliptic framework.
In addition to $\Gamma$, we assume that $\M$ is endowed with another smooth symmetric bilinear differential form,
indicated with $\Gamma^Z$, satisfying for $f,g \in C^\infty(\M)$
\[
 \Gamma^Z(fg,h) = f\Gamma^Z(g,h) + g \Gamma^Z(f,h),
\]
and $\Gamma^Z(f) = \Gamma^Z(f,f) \ge 0$.

\
We make the following assumptions that will be in force throughout the paper:

\begin{itemize}
\item[(H.1)] There exists an increasing
sequence $h_k\in C^\infty_0(\bM)$   such that $h_k\nearrow 1$ on
$\bM$, and \[
||\Gamma (h_k)||_{\infty} +||\Gamma^Z (h_k)||_{\infty}  \to 0,\ \ \text{as} \ k\to \infty.
\]
\item[(H.2)]
For any $f \in C^\infty(\bM)$ one has
\[
\Gamma(f, \Gamma^Z(f))=\Gamma^Z( f, \Gamma(f)).
\]
\item[(H.3)]
For every $t \ge 0$, $P_t 1=1$ and for every $f \in C_0^\infty(\bM)$ and $T \ge 0$, one has
\[
\sup_{t \in [0,T]} \| \Gamma(P_t f)  \|_{ \infty}+\| \Gamma^Z(P_t f) \|_{ \infty} < +\infty,
\]
where $P_t$ is the heat semigroup generated by $L$.

\end{itemize}

\
(Details about the assumptions are discussed in \cite{BG}) The assumption (H.1) is implied by the completeness of the metric space. In the sub-Riemannian geometries covered by the present work, the assumption (H.2) means that the torsion of the sub-Riemannian connection is vertical (for instance, Sasakian condition of CR manifolds). Removing this assumption in certain cases is discussed in \cite{BauWa}. Assumption (H.3) is necessary to rigorously justify the Bakry-\'Emery type arguments. It is a consequence of the generalized curvature dimension inequality below in many examples (see \cite{BG}).

\
In addition to $\Gamma$ and $\Gamma^Z$, we denote the following second order differential bilinear forms: for any $f,g\in C^\infty(\M)$,
\begin{equation}
\Gamma_{2}(f,g) = \frac{1}{2}\big[L\Gamma(f,g) - \Gamma(f,Lg)-\Gamma (g,Lf)\big],
\end{equation}
\begin{equation}
\Gamma^Z_{2}(f,g) = \frac{1}{2}\big[L\Gamma^Z (f,g) - \Gamma^Z(f,Lg)-\Gamma^Z (g,Lf)\big].
\end{equation}
As for $\Gamma$ and $\Gamma^Z$, we denote $\Gamma_2(f) = \Gamma_2(f,f)$, $\Gamma_2^Z(f) = \Gamma^Z_2(f,f)$.

The following curvature dimension condition was introduced in \cite{BG}.

\begin{definition}[\cite{BG}, generalized curvature dimension inequality] \label{GCDI}
We say that $L$ satisfies the \emph{generalized curvature dimension inequality} \emph{CD}$(\rho_1,\rho_2,\kappa,d)$ on $\M$ if
there exist constants $\rho_1 \in \mathbb{R} $,  $\rho_2 >0$, $\kappa \ge 0$, and $0< d < \infty$ such that the inequality
\begin{equation*}
\Gamma_2(f) +\nu \Gamma_2^Z(f) \ge \frac{1}{d} (Lf)^2 +\left( \rho_1 -\frac{\kappa}{\nu} \right) \Gamma(f) +\rho_2 \Gamma^Z(f)
\end{equation*}
 holds for every  $f\in C^\infty(\bM)$ and every $\nu>0$.
\end{definition}

\
The inequality $CD(\rho_1,\rho_2,\kappa,d)$ turns out to be equivalent to lower bounds on intrinsic curvature tensors in \cite{BG}.
The following is an exemplary curvature condition implying $CD(\rho_1,\rho_2,\kappa,d)$ on CR manifold.

\begin{proposition}\cite{BG}
Let $(\bM,\theta)$ be a complete \emph{CR} Sasakian manifold  with real dimension $2n+1$.
The Tanaka-Webster Ricci tensor satisfies the bound
\[
 \emph{Ric}_x(v,v)\ \ge \rho_1|v|^2, \forall x \in \M, \forall v \in \mathcal H_x,
\]
if and only if the curvature dimension inequality \emph{CD}$(\rho_1,\frac{d}{4},1,d)$ holds with $d = 2n$ and $\Gamma^Z(f)=(Tf)^2$ and the hypothesis (H.1),(H.2),(H.3) are satisfied.
\end{proposition}

\
With the curvature inequality condition assumed, various aspects on sub-Riemannian manifolds have been discovered in \cite{BG},\cite{BB},\cite{BBG},\cite{BBGM},\cite{BG_Riesz},\cite{BauWa},\cite{BauKim}. In particular, we have the following essential properties -
two-sided heat kernel bounds and volume doubling property of balls with exponential term.

\begin{proposition}
\cite{BBGM} If we assume the curvature condition $CD(-K,\rho_2,\kappa,d)$, $K>0$ on $\M$, for any $x,y \in \M$, $t>0$, $r>0$,
\begin{align}
&\label{ineq:hk LB}\quad p_t(x,y) \geq \frac{C_{1}}{\mu(B(x,\sqrt t))} \exp\left(- \frac{D}{2d} \frac{d(x,y)^2}{t} - C_{2} K( t+  d(x,y)^2 ) \right) \\
&\label{ineq:hk UB}\quad p_t(x,y) \leq \frac{C_{3}}{\mu(B(x,\sqrt t))^{1/2} \mu(B(y,\sqrt t))^{1/2}} \exp\left( C_{4} Kt-\frac{d(x,y)^2}{5t} \right) \\
&\label{ineq:doubling}\quad \mu(B(x,2r)) \leq C_{d1} \exp(C_{d2} K r^2) \mu(B(x,r)) ,
\end{align}
where $D=d\left(1+\frac{3\kappa}{2\rho_2}\right)$, $C_1,C_2,C_3,C_4,C_{d1},C_{d2}$ are positive and determined by $\rho_2,\kappa,d$.
\end{proposition}

\
Denote $Q=\log_2 C_{d1}$. (\ref{ineq:doubling}) implies that for any $\lambda>1$,
\begin{align}
 \begin{aligned} \label{ineq:doubling2}
 \frac{\mu(B(x,\lambda r))}{\mu(B(x,r))}
   &\leq C_{d1} ^{\lceil \log_2 \lambda \rceil} \exp(C_{d2} K \sum_{i=0}^{\lceil \log_2 \lambda \rceil-1 } (2^i r)^2 ) \\
   &\leq  C_{d1} \lambda^Q \exp(\frac{4 C_{d2} }{3}  K  (\lambda r)^2 ) .
 \end{aligned}
\end{align}

\
This doubling property allows us to estimate $\mu(B(x,\sqrt t))$ by $\mu(B(y,\sqrt t))$:
\begin{align*}
 \mu(B(x,\sqrt t)) &  \leq \mu(B(y,\sqrt t + d(x,y)))  \\
 & \leq \mu(B(y,\sqrt t)) 2C_{d1} \left( 1+ \frac{d(x,y)^2}{t} \right)^{Q/2}  \exp\left(\frac{8C_{d2}}{3} K( t+d(x,y)^2 ) \right) .
\end{align*}

So we modify the upper bound of heat kernel (\ref{ineq:hk UB}) with the volume of a single ball, i.e., for $C_5,C_6>0$ depending on $\rho_2,\kappa,d$,
\begin{align}
\begin{aligned} \label{ineq:hk UB2}
 p_t(x,y) 
 & \leq \frac{C_{5}}{\mu(B(x,\sqrt t)) } \exp\left( C_{6} K(t+d(x,y)^2)-\frac{d(x,y)^2}{6t} \right).
\end{aligned}
\end{align}
Note that $ 1+A \leq C(\epsilon) e^{\epsilon A} $ for $\forall \epsilon>0, A\geq 0$ is applied.

\begin{remark}
As mentioned in \cite{BBGM}, the square in the exponent of volume doubling property might not be optimal.
For instance, in the Riemannian manifold with Ricci tensor bounded below by $-K<0$,
by the Bishop-Gromov comparison theorem we have $V(x,\lambda r)\leq V(x,r) \lambda^n \exp(\sqrt{(n-1)K} (\lambda r) ) $ where $\lambda>1$ and
$V(x,r)$ is the Riemannian measure of the ball $B(x,r)$.

\
This yields the difference of the exponent in (\ref{ineq:Buser}) and (\ref{ineq:Buser subelliptic}).
\end{remark}

\begin{remark} \label{rmk:A epsilon}
(\cite{BB}) Notice that the positive solution $u$ carries additional condition in the Li-Yau type inequality, Corollary \ref{Cor:LiYau}. 
Due to the technical reason in the proof of Theorem 6.1 in \cite{BG}, $u$ needs to be contained in 
$\mathcal A_\epsilon=\{f\in C_b^\infty(\M) : f-\epsilon\geq 0, \sqrt{\Gamma(f-\epsilon)},\sqrt{\Gamma^Z(f-\epsilon)}\in L^2(\M) \}$. 
Same restriction is required for log-Sobolev inequality in \cite{BB}.
\end{remark}

\section{Poincar\'e inequality on the ball}
\subsection{Lower bound of the Dirichlet heat kernel on the ball}

\
Throughout this section, $L$ satisfies $CD(-K,\rho_2,\kappa,d)$, $K>0$ on $\M$.

\
To adapt Kusuoka and Stroock's idea \cite{KS}, the necessary ingredients will be two-sided heat kernel bound (\ref{ineq:hk LB}),(\ref{ineq:hk UB})
and doubling (\ref{ineq:doubling2}).


\
Denote $B=B(x_0,r)$, sub-Riemannian ball centered at $x_0$ with radius $r$. On the ball $B$, the Dirichlet heat kernel $p^{B,D}_t(x,y)$ will be defined by
the transition probability
\[
 p^{B,D}_t(x,y)d\mu(y) = P[ \zeta>t, X(t) \in d\mu(y)] ,
\]
where $X(t)$ is the associated Markov process of the semigroup operator $P_t = e^{tL}$ with $X(0)=x$,
and the lifetime of $X$ in $B$ is $\zeta = \inf\{ t>0, X(t) \not\in B \}$.

\
First, the lower bound of the Dirichlet heat kernel for close $x,y$ can be obtained by the argument of Kusuoka and Stroock \cite{KS}:

\begin{lemma}
For any $k\in (0,1)$, there exists $C_\alpha = C(k,\rho_2,\kappa,d) \in (1,\infty)$ such that for any $x_0 \in \M$, $r>0$ and $\alpha=\sqrt{ \frac{1}{C_\alpha(Kr^2+1)} } \in (0,1)$,
the Dirichlet heat kernel on $B=B(x_0,r)$ has lower bound
\begin{align} \label{eq:dirichlet local bound}
 p^{B,D}_t (x,y) \geq \frac{c}{\mu(B(x,\sqrt t))} \exp\left( -C\frac{d(x,y)^2}{t}  \right)
\end{align}
for all $t\in(0,(\alpha r)^2 ]$ and $x,y \in B(x_0,kr)$ such that $d(x,y) \leq \alpha r$. \\
Here $c,C >0$ depend only on $\rho_2,\kappa,d$.
\end{lemma}

\begin{proof}
Let $\alpha=\left( C_\alpha(Kr^2+1) \right)^{-\frac{1}{2}} \in (0,1)$ with some $C_\alpha>1$ which will be determined later.
Note that $ K ( \alpha r )^2 \leq C_\alpha^{-1} \leq 1 $ .

\
Let $d(x,y)\leq \alpha r$ and $t \leq (\alpha r)^2$.
The Dirichlet heat kernel can be written by the heat kernel of $\M$ and the lifetime of the process in the domain. That is,
\begin{align*}
 p^{B,D}_t (x,y) & = p_t(x,y) - \mathbb E ^x[ p_{t-\zeta} ( X(\zeta),y) , \zeta<t ], \quad \zeta = \inf\{ t>0, X(t) \not\in B(x_0,r) \}.
\end{align*}

The lower bound (\ref{ineq:hk LB}) on the heat kernel $p_t(x,y)$ over the whole manifold yields
\begin{align*}
  p_t(x,y) 
   \geq & \frac{C_1 e^{-2 C_2} }{\mu(B(x,\sqrt t))} \exp\left( -\frac{D}{2d}\frac{d(x,y)^2}{t}  \right) .
\end{align*}

If we use upper bound (\ref{ineq:hk UB}) on the heat kernel in the expectation, we have
\[
 p_{t-\zeta}(X(\zeta),y) \leq \frac{C_3 \exp\left( C_4 K (t-\zeta)-\frac{d(X(\zeta),y)^2}{5(t-\zeta)} \right)}
                                            {\mu(B(X(\zeta),\sqrt{t-\zeta}))^{1/2} \mu(B(y,\sqrt{t-\zeta}))^{1/2}} .
\]

The balls can be replaced by concentric balls using the doubling property (\ref{ineq:doubling2}) as follows.
\begin{align*}
\frac{\mu(B(x,\sqrt t))}{\mu(B(X(\zeta),\sqrt{t-\zeta}))} \leq & \frac{\mu(B(X(\zeta),3r))}{\mu(B(X(\zeta),\sqrt{t-\zeta}))}
 \leq  C_{d1}  \left( \frac{3r}{\sqrt{t-\zeta}} \right)^Q \exp( \frac{4C_{d2} K}{3} (3r)^2 )  , \\
\frac{\mu(B(x,\sqrt t))}{\mu(B( y ,\sqrt{t-\zeta}))} \leq & \frac{\mu(B(y,2\alpha r))}{\mu(B(y,\sqrt{t-\zeta}))}
 \leq  C_{d1}  \left( \frac{2\alpha r}{\sqrt{t-\zeta}} \right)^Q \exp( \frac{4C_{d2} K}{3} (2\alpha r)^2 ).
\end{align*}

With $ d(X(\zeta),y ) \geq r(1-k)$ and $t-\zeta \leq t\leq (\alpha r)^2 $, the above controls imply that
\begin{align*}
\mathbb E ^x[ p_{t-\zeta} & ( X(\zeta),y) , \zeta<t ] \\
 \leq &  \frac{C_3 C_{d1}e^{3C_{d2}+C_4} }{\mu(B(x,\sqrt t))} \exp\left(-\frac{r^2(1-k)^2}{10 t}+ 6C_{d2} K r^2\right) \cdot C_\nu \exp\left( -\nu \frac{1}{\alpha^2}\right).
\end{align*}

The last term came from $ \E ^x \left[  \left( \frac{\sqrt{(3r)(2\alpha r)}}{\sqrt{t-\zeta}} \right)^Q  \exp\left(  -\frac{r^2(1-k)^2}{10 (t-\zeta)}  \right) \right]
\leq C_\nu \exp\left( -\nu \frac{1}{\alpha^2}\right)$,
which holds if we choose $C_\nu \geq \left( \frac{60 Q}{e(1-k)^2} \right)^{Q/2}$, $\nu\leq \frac{(1-k)^2}{20}$.\\

\
Combining these upper and lower estimates with $t\leq (\alpha r)^2$, $d(x,y) \leq \alpha r $,
\begin{align*}
 & p^{B,D}_t (x,y) \\
 & \geq  \frac{C_1 e^{-2 C_2}}{\mu(B(x,\sqrt t))} e^{\left( -\frac{D}{2d}\frac{d(x,y)^2}{t} \right)}
  \left[ 1 - C \exp\left(-\frac{r^2}{t}( \frac{(1-k)^2}{10} - \frac{D}{2d}\alpha^2 )+ 6C_{d2} K r^2  - \frac{\nu}{\alpha^2} \right)  \right] ,
\end{align*}
where $C= C_3 C_{d1} C_\nu C_1^{-1} e^{3C_{d2}+C_4+2 C_2}$.

\
Choose $\alpha$ small enough for $ \left[ 1 - C \exp\left(-\frac{r^2}{t}( \frac{(1-k)^2}{10} - \frac{D}{2d}\alpha^2 )+ 6C_{d2} K r^2  - \frac{\nu}{\alpha^2} \right)  \right] \geq \frac{1}{2}$.

\
Then we conclude
\[
 p^{B,D}_t (x,y) \geq \frac{c}{\mu(B(x,\sqrt t))} \exp\left( -C\frac{d(x,y)^2}{t}  \right) ,
\]
for all $t\leq (\alpha r)^2$, $d(x,y)\leq \alpha r$, where $c=2^{-1}C_1 e^{-2C_2} ,C=\frac{D}{2d} > 0$ depend on $\rho_2,\kappa,d$.\\



\
For instance, if we pick large $C_\alpha = C(\rho_2,\kappa,d, k )>0$ such as
\[
 C_\alpha \geq \max\left\{ \frac{\frac{D}{2d}+\ln(2C_3 C_{d1} C_\nu C_1^{-1} e^{3C_{d2}+C_4+2 C_2})+6C_{d2}}{\nu +\frac{(1-k)^2}{10}}, \left(\frac{2d}{D} \frac{(1-k)^2}{10}\right)^{-1} \right\} ,
\]
then our choice of
\begin{align}\label{eq:choice of alpha}
 \alpha^2 = \frac{1}{C_\alpha (Kr^2+1)} ,
\end{align}
satisfies the estimates above.

\end{proof}

\
Next step is the lower bound of Dirichlet heat kernel for any $x,y$ in the smaller ball which is followed by the chain argument.
Note that our lemma holds for any $r>0$ with the exponential square of radius, while the classic lemma holds only for $0<r \leq 1$.

\begin{lemma}
For any $0<k<1$ and $0<\delta<1$, there exists $0<c<1$, $C>0$ such that for any $x_0\in \M$ and $r>0$, the Dirichlet heat kernel on the ball $B=B(x_0,r)$ has lower bound
\begin{align*}
 p_t^{B,D} (x,y) \geq \frac{c \exp(-C Kr^2)}{\mu(B(x_0,kr))}
\end{align*}
for all $x,y\in B(x_0,kr)$ and $\delta r^2 \leq t \leq r^2 $ .
\end{lemma}

\begin{proof}
Choosing $\alpha \in (0,1) $ of (\ref{eq:choice of alpha}) in the previous lemma, for all $t\leq (\alpha r)^2$, $x,y\in B(x_0,kr)$, $d(x,y)\leq \alpha r$,
\begin{align*}
 p^{B,D}_t (x,y) \geq & \frac{c}{\mu(B(x,\sqrt t))} \exp( -C\frac{d(x,y)^2}{t} ).
\end{align*}

\
Now let $x,y$ be any points in $ B(x_0,kr) $ and $\delta r^2 \leq t \leq r^2$. Set $ n = \lceil 16\alpha^{-2} \rceil $, then $16 \alpha^{-2} \leq n \leq 17 \alpha^{-2}$.

\
We choose $\{\xi_i\}_{i=0,1,\cdots,2n} \subset B(x_0,kr)$ such that
\begin{align*}
 &\xi_0=x, \quad \xi_n = x_0,\quad \xi_{2n} = y, \\
 &d(\xi_k,\xi_{k+1}) \leq \frac{r}{n} \leq \frac{\alpha r}{4}.
\end{align*}

\
Let $\tau = \frac{t}{2n} $. Since $\sqrt{\tau} \leq \frac{\alpha r}{4}$, if $\eta_k \in B(\xi_k , \sqrt{\tau})$, then $d(\eta_k,\eta_{k+1})\leq \alpha r$.\\

\
By the previous lemma,
\begin{align*}
p^{B,D}_\tau (\eta_k,\eta_{k+1}) \geq & \frac{c}{\mu(B(\eta_k,\sqrt \tau ))} \exp( -C\frac{d(\eta_k,\eta_{k+1})^2}{\tau}  ) \\
\geq & \frac{c C_{d1}^{-1} \exp (-C_{d2} K (\alpha r)^2)}{\mu(B(\xi_k,\sqrt \tau ))} \exp( -C\frac{d(\eta_k,\eta_{k+1})^2}{\tau}  ) .
\end{align*}

\
And we see that
\begin{align*}
 \frac{d(\eta_k,\eta_{k+1})^2}{\tau} \leq \left(\frac{d(\xi_k,\xi_{k+1})}{\sqrt{\tau}}+2 \right)^2 \leq
  \left(\frac{r\sqrt{2}}{\sqrt{nt} }  +2 \right)^2 \leq \frac{4}{\delta n} + 8 .
\end{align*}

Observing
\begin{align*}
 p^{B,D}_t(x,y) \geq \int_{B(\xi_{2n-1},\sqrt{\tau})} & \cdots \int_{B(\xi_1,\sqrt{\tau})} p^{B,D}_\tau(x,\eta_1) \\
    & \cdot p^{B,D}_\tau(\eta_1,\eta_2) \cdots p^{B,D}_\tau(\eta_{2n-1},y) d\eta_1 \cdots d\eta_{2n-1},
\end{align*}
we obtain
\begin{align*}
 p^{B,D}_t(x,y) \geq \frac{1}{\mu(B(x_0,\sqrt{\tau}))} \left(c C_{d1}^{-1} \exp(-C(\frac{4}{\delta n} + 8)-C_{d2}K(\alpha r)^2 ) \right)^{2n} .
\end{align*}

\
Doubling property (\ref{ineq:doubling2}) yields
\begin{align*}
 \frac{1}{\mu(B(x_0,\sqrt{\tau}))} \geq \frac{1}{\mu(B(x_0,r ))} \geq \frac{C_{d1}^{-1} k^Q \exp(-\frac{4C_{d2}}{3} Kr^2) }{\mu(B(x_0,kr))}  .
\end{align*}

\
Also since $c C_{d1}^{-1} \exp(-8C) <1$ and $n\leq 17 \alpha^{-2} = 17 C_\alpha (K r^2 +1 )$ from (\ref{eq:choice of alpha}),
\begin{align*}
 &\left(c C_{d1}^{-1} \exp\left(-C(\frac{4}{\delta n} + 8)-C_{d2}K(\alpha r)^2 \right) \right)^{2n} \\
 &\geq \exp\left(-\frac{8C}{\delta}  - (8C-\ln(c C_{d1}^{-1}) +C_{d2})\cdot 34 C_\alpha (Kr^2 +1) \right) .
\end{align*}
This concludes our lemma
\[
 p_t^{B,D} (x,y) \geq \frac{c' \exp(-C' Kr^2)}{\mu(B(x_0,kr))},
\]
where
\begin{align*}
&c' =C_{d1}^{-1} k^Q \exp\left(-\frac{8C}{\delta}  - (8C-\ln(c C_{d1}^{-1}) +C_{d2})(34 C_\alpha)  \right)  , \\
&C' = \frac{4C_{d2}}{3} + (8C-\ln(c C_{d1}^{-1}) +C_{d2}) (34 C_\alpha )
\end{align*}
are determined by $\rho_2,\kappa,d,k,\delta$.
\end{proof}

\subsection{proof of Theorem \ref{thm:buser subelliptic}}

\
In this section, we utilize Dirichlet, Neumann heat semigroup which can be found in \cite{St2},\cite{SC},\cite{KS},\cite{GySa}, then we will follow the arguments in \cite{KS}
to prove Poincar\'e inequality (\ref{ineq:Buser subelliptic}).

\
Let $B=B(x_0,r)$. Define a subspace $D^\infty \subset C^\infty(B)$ as a collection
of functions $f$ satisfying $ -\int_B g Lf d\mu = \int_B \Gamma(g,f)d\mu $ for $\forall g \in C^\infty(B)$. Note that $C_0^\infty(B) \subset D^\infty\subset C^\infty(B)$.\\
The Dirichlet form $\mathcal E (f,g) = \int_B \Gamma(f,g)d\mu $ on $D^\infty$ is closable in $L^2(B)$, and by closing it we gain a Dirichlet form and associated
Markov heat semigroup $P_t^{B,N}$ with Neumann boundary condition.

\
If we denote $p_t^{B,N}$ by the Neumann heat kernel over $B$, it will be a smooth kernel of the Neumann heat semigroup and its associated transition probability function.
Naturally, since $C_0^\infty(B) \subset D^\infty$, the Neumann heat kernel dominates the Dirichlet heat kernel, i.e., $p_t^{B,N} \geq p_t^{B,D}$.


\begin{proof}[Proof of Theorem \ref{thm:buser subelliptic} .]
We will prove the inequality with $B(x_0,r/2)$ on the left hand side. Then by the Whitney type covering lemma (section 5 in \cite{Je}),
we can match the balls on the both sides. The Whitney decomposition only requires a doubling property in the domain of argument.
In $B(x_0,10 r)$, the doubling property holds with fixed constant $C_{d1} \exp(C_{d2}K(10 r)^2)$, which will be multiplied at the end following the argument.

\
From the previous lemma, for $x,y\in B(x_0,r/2)$,
\begin{align*}
p_{r^2}^{B(x_0,r),N}(x,y) \geq \frac{c e^{-C K r^2} }{\mu(B(x_0,r/2))}  .
\end{align*}
For any $f\in C^\infty(B)$ and $x\in B(x_0,r/2)$,
\begin{align*}
P_{r^2}^{B(x_0,r),N}(f - & P_{r^2}^{B(x_0,r),N} f (x) )^2 (x) \\
& \geq \frac{c e^{-C K r^2} }{\mu(B(x_0,r/2))} \int_{B(x_0,r/2)} (f(y)-  P_{r^2}^{B(x_0,r),N} f (x) )^2 d\mu(y) \\
 &\geq \frac{c e^{-C K r^2} }{\mu(B(x_0,r/2))} \int_{B(x_0,r/2)} (f(y)-  f_{B(x_0,r/2)} )^2 d\mu(y).
\end{align*}
On the other hand,
\begin{align*}
\int_{B(x_0,r)} P_{r^2}^{B(x_0,r),N}(f - & P_{r^2}^{B(x_0,r),N} f (x) )^2 (x) d\mu(x) \leq \int_{B(x_0,r)} (f^2 - P_{r^2}^{B(x_0,r),N} f (x) ^2 )d\mu(x) \\
 &=  \int_0^{r^2} \int_{B(x_0,r)}  -\frac{d}{dt} ( P_t^{B(x_0,r),N} f (x)) ^2 d\mu(x) dt  \\
 &= \int_0^{r^2} \int_{B(x_0,r)} -2 P_t^{B(x_0,r),N} f (x) L P_t^{B(x_0,r),N} f (x)  d\mu(x) dt \\
 &= \int_0^{r^2} \int_{B(x_0,r)} 2 \Gamma(  P_t^{B(x_0,r),N} f (x) )  d\mu(x) dt \\
 &\leq 2 r^2 \int_{B(x_0,r)} \Gamma( f )  d\mu ,
\end{align*}
where the last inequality comes from $\frac{d}{dt} \Gamma(P_t f) \leq 0$. And we obtain our desired conclusion
\begin{align*}
 \int_{B(x_0,r/2)} (f(x)-f_{B(x_0,r/2)} )^2 d\mu(x) \leq C'_{p1} r^2 e^{C'_{p2} Kr^2} \int_{B(x_0,r)} \Gamma(f) d\mu ,
\end{align*}
with $C'_{p1}=2/c, C'_{p2}=C$.
\end{proof}

\subsection{Sobolev inequality and $L^p$ mean value estimate }

\
As a consequence of Theorem \ref{thm:buser subelliptic}, this section is dedicated to the Sobolev inequality and $L^p$ mean value inequality for subharmonic functions,
which will be essential to prove Theorem \ref{thm:uniq of positive} and \ref{thm:uniq of Lp}. Throughout this paper, harmonic (resp.subharmonic) functions are $f\in \Dom(L)$ satisfying $Lf=0$ (resp.$Lf\geq 0$).


\
Assume that $L$ satisfies $CD(-K,\rho_2, \kappa,d)$, $K>0$ on $\M$.
We have Poincar\'e inequality (\ref{ineq:Buser subelliptic}) and exponential doubling property (\ref{ineq:doubling}).


\
With these two ingredients, one can derive local Sobolev inequality in \cite{SC2},\cite{SC3}. This is a classic path to Moser's iteration for Harnack's inequality.
See theorem 2.2 in \cite{SC2} and section.10 in \cite{SC3} (also the last section of \cite{Va}).

\
Note that in \cite{GN2}, one can find that the weak $L^1$ Poincar\'e inequality with the doubling property derives the isoperimetric inequality
(and Sobolev inequalities) in Carnot-Carath\'eodory spaces.

\begin{proposition}[\cite{SC2},\cite{SC3}, Sobolev inequality on balls]
If the Poincar\'e inequality (\ref{ineq:Buser subelliptic}) and the volume doubling condition (\ref{ineq:doubling}) are satisfied for any $r>0$,
then for any $x\in \M$, $0<r$, $ f \in C_0^\infty (B(x,r))$, denoting $B=B(x,r)$,
\begin{align} \label{local sobolev inequality}
 \left( \frac{1}{\mu(B)} \int_{B} |f|^{\frac{2Q}{Q-2}} d\mu \right)^{\frac{Q-2}{2Q}} \leq C_1 r e^{C_e K r^2} \left( \frac{1}{\mu(B)}
 \int_{B} (\Gamma(f)+r^{-2} |f|^2 )d\mu \right)^{\frac{1}{2}},
\end{align}
where $Q=\log_2 C_{d1}$ in (\ref{ineq:doubling}), $C_1, C_e >0$ depend only on $\rho_2,\kappa,d$.
\end{proposition}

\
Note that this Sobolev inequality can also be obtained by following steps of \cite{SC3}.\\

\
With $CD(-K,\rho_2,\kappa,d)$ assumed, by the upper bound of the heat kernel (\ref{ineq:hk UB2})
\begin{align*}
&p_t^{B,D}(x,y)\leq p_t(x,y) \leq \frac{C_{5}}{\mu(B(x,\sqrt t)) } \exp\left( C_{6} K(t+d(x,y)^2)-\frac{d(x,y)^2}{6t} \right),
\end{align*}
where $B=B(x_0,r)$. Since $0<t\leq r^2$ and $d(x,y)\leq 2r$ for $x,y\in B$, by (\ref{ineq:doubling2})
\begin{align*}
 \mu(B(x_0,r )) \leq \mu(B(x,2r)) \leq C \left( \frac{r }{\sqrt t }\right)^Q \exp(c Kr^2 )    \mu(B(x,\sqrt t  )).
\end{align*}
Therefore, the Dirichlet heat kernel will be bounded from above by
\begin{align*}
 &p_t^{B,D}(x,y) \leq \frac{C}{\mu(B(x,\sqrt t)) } e^{ c Kr^2 } \leq \frac{C'}{\mu(B(x_0,r)) }  r^{Q} t^{-Q/2} e^{ c' Kr^2 } .
\end{align*}

\
Proposition 10.1 in \cite{SC3} (also \cite{Va}) states that
\begin{align*}
 & \| P_t^{B,D}  \|_{1\rightarrow \infty} \leq C_0 t^{-Q/2}, \quad \forall 0< t< t_0\\
 & \quad \Longrightarrow \|f \|_{\frac{2Q}{Q-2}} \leq C_0^{1/Q} \left( C \| \sqrt{ \Gamma(f)}\|_2 + t_0^{-1/2} \|f \|_2 \right) ,\quad \forall f \in C_0^\infty(B).
\end{align*}
Taking $C_0 = \frac{C'}{\mu(B(x_0,r)) }  r^{Q} e^{ c' Kr^2 } $ and $t_0 = r^2 $ with $A+B\leq (2A^2 +2B^2 )^{1/2} $,
the Sobolev inequality (\ref{local sobolev inequality}) is proved.\\

\
Once the local Sobolev embedding is acquired, our goal of this section, $L^p$ mean value estimate, can be obtained through
the Moser's iteration. One can find the arguments for the Riemannian case in \cite{SC}.

\
In \cite{SC3},\cite{SC2}, the author obtained parabolic $L^p$ mean value estimate. But in our context, $L^p$ mean value estimate for subsolution of $L$-Laplace equation
will be enough.
\begin{lemma} [One step of Moser's iteration] \label{lem:moser_1step}
We assume that $\M$ satisfies $CD(-K,\rho_2,\kappa,d)$, $K>0$.
For any subharmonic function $u(x)\geq0$, i.e. $Lu(x)\geq 0$, and $0< R_1 < R_2 \leq R$, $p\geq2$,
\begin{align} \label{eq:moser 1step}
 \int_{ B(R_1)} u^{p\theta} d\mu \leq C_2 e^{2C_e K R^2} \frac{R^2} { (R_2-R_1)^{2}} V^{1-\theta} \left(  \int_{ B(R_2)}u^p d\mu \right)^\theta ,
\end{align}
where $\theta=1+\frac{2}{Q}$, $B(\cdot)=B(x_0,\cdot)$, $V=\mu(B(R))$.
\end{lemma}

\begin{remark} \label{remark:cutoff}
For any $0<R_1<R_2<\infty$, there exists a Lipschitz continuous cut-off function $\psi \geq 0$ satisfying $\psi |_{B(R_1)} = 1$, $\supp \psi \subset B(R_2)$,
$\sqrt{\Gamma(\psi)} \leq \frac{C}{R_2-R_1}$ almost everywhere for some $C>0$ which is independent to $R_1, R_2$.
See theorem 1.5 in \cite{GN}, lemma 3.6 in \cite{CGL} and \cite{St1}.
\end{remark}


\begin{proof}

Denote $B=B(R)$. Let $\psi \geq 0$ be a cut-off function satisfying $\psi |_{B(R_1)} = 1, \supp \psi \subset B(R_2)$,
$\sqrt{\Gamma(\psi)} \leq \frac{C}{R_2-R_1}$ almost everywhere.
 Choose a test function $\phi = \psi^2 u^{p-1} \geq 0 $, and the subharmonicity condition of $u$ implies
\[
 0\leq (Lu,\phi) = \int_B -\Gamma(u,\psi^2 u^{p-1}) d\mu = \int_B \left( -(p-1) \psi^2 u^{p-2} \Gamma(u) -2\psi u^{p-1} \Gamma(u,\psi) \right) d\mu.
\]
Hence, by Cauchy-Schwarz inequality
\begin{align*}
 \frac{p-1}{2}\int_B \psi^2 u^{p-2}\Gamma(u) d\mu &\leq \int_B \psi u^{p-1} \sqrt{\Gamma(u) \Gamma(\psi)} d\mu \\
  \leq & \left(\int_B \psi^2 u^{p-2} \Gamma(u) d\mu \right)^{1/2} \left(\int_B u^p \Gamma(\psi) d\mu \right)^{1/2} .
\end{align*}
So, we have
\begin{align} \label{eq:pf_moser_1}
 \frac{(p-1)^2}{4}\int_B \psi^2 u^{p-2}\Gamma(u) d\mu \leq \int_B u^p \Gamma(\psi) d\mu .
\end{align}

On the other hand, if we apply H\"{o}lder inequality and the local Sobolev inequality (\ref{local sobolev inequality}) on $\psi u^{p/2}$, it will give
\begin{align*}
 &\frac{1}{\mu(B)} \int_B  |\psi u^{p/2}|^{ 2(1+\frac{2}{Q})} d\mu
   \leq \left( \frac{1}{\mu(B)} \int_B |\psi u^{p/2}|^{\frac{2Q}{Q-2}}  d\mu  \right)^{\frac{Q-2}{Q}}
         \left( \frac{1}{\mu(B)} \int_B |\psi u^{p/2}|^{2} d\mu \right)^{\frac{2}{Q}} \\
  &\leq  C_1^2 R^2 e^{2C_e K R^2}\left( \frac{1}{\mu(B)} \int_B (\Gamma(\psi u^{p/2} )+R^{-2}|\psi u^{p/2}|^2 ) d\mu \right)
                 \left( \frac{1}{\mu(B)}\int_B \psi^2 u^p d\mu \right)^{\frac{2}{Q}} \\
  &\leq  C_1^2 R^2 e^{2C_e K R^2}\left( \frac{1}{\mu(B)} \int_B \Gamma(\psi u^{p/2} ) d\mu+R^{-2}\frac{1}{\mu(B)} \int_{\supp(\psi)} u^p d\mu  \right) \\
 &\quad\quad \cdot  \|  \psi\|_\infty^{\frac{4}{Q}} \left(\frac{1}{\mu(B)} \int_{\supp(\psi)} u^p d\mu \right)^{\frac{2}{Q}} .
\end{align*}

Using (\ref{eq:pf_moser_1}), the gradient term can be written by
\begin{align*}
  & \int_B  \Gamma(\psi u^{p/2} ) d\mu  \leq \int_B 2\left( u^p \Gamma(\psi) +\frac{p^2}{4}\psi^2 u^{p-2}\Gamma(u) \right) d\mu   \\
  &\leq \left( 2+ \frac{2p^2}{(p-1)^2}\right) \int_B u^p \Gamma(\psi) d\mu
    \leq \left( 2+ \frac{2p^2}{(p-1)^2}\right) \|\Gamma(\psi) \|_\infty \int_{\supp \psi} u^p  d\mu.
\end{align*}

Therefore, given $\supp \psi \subset B(R_2)$, $\psi|_{B(R_1)}=1$, $0\leq \psi \leq1$ and $\| \Gamma(\psi) \|_\infty \leq \frac{C^2}{(R_2-R_1)^2}$, we obtain
\begin{align*}
 &\frac{1}{\mu(B)} \int_{B(R_1)} u^{p \theta} d\mu \\
 &\leq  C_1^2 e^{2C_e K R^2}\left(\left( 2+ \frac{2p^2}{(p-1)^2}\right) \frac{C^2 R^2}{ (R_2-R_1)^{2}} +1 \right)\left(\frac{1}{\mu(B)} \int_{B(R_2)} u^p d\mu \right)^{\theta} \\ &\leq  11 C^2 C_1^2  e^{2C_e K R^2}\frac{ R^2 } { (R_2-R_1)^{2}} \left(\frac{1}{\mu(B)} \int_{B(R_2)} u^p d\mu \right)^{\theta} ,
\end{align*}
where $\theta = 1+\frac{2}{Q}$. Note that we can assume that $C>1$ without loss of generality.\\
The desired inequality (\ref{eq:moser 1step}) is proved with $C_2 = 11 C^2 C_1^2$.
\end{proof}

\
Now by iterating the above lemma, we prove $L^p$ mean value estimate.

\begin{theorem} [$L^p$ mean value inequality, $p\geq 2$] \label{thm:Lp mean value p>2}
For any $0<\delta<1$, any $p\geq 2$, and any non-negative subsolution $u$ of $Lu=0$ in a ball $B(R)$ of volume $V$,
\begin{align} \label{ineq:Lp mean value p>2}
 \sup_{\delta B} \{ u^p\} \leq C_3 e^{Q C_e K R^2} (1-\delta)^{-Q} \left( V^{-1} \int_B u^p d\mu \right) .
\end{align}
\end{theorem}

\begin{proof}
\
For $i=0,1,2, \cdots $, set $p_i = p \theta ^i $ where $\theta = 1+\frac{2}{Q}$. \\
And let $R_0=R$, $R_{i} - R_{i+1} = \frac{  (1-\delta)R} {2 ^{(i+1)}}  $, i.e.
\[
 R_i = R-\sum_{j=1}^i \frac{  (1-\delta)R} {2 ^{j}} =  R- (1-\delta)R ( 1 - \frac{1}{2^i} ) = \delta R + \frac{(1-\delta)R}{2^i} .
\]
By Lemma \ref{lem:moser_1step},
\[
 \int_{B(R_{i+1})} u^{p_{i+1}} d\mu \leq  \frac{C_2 2^{2(i+1)} V^{1-\theta} }{(1-\delta)^{2}} e^{2C_e K R^2} \left( \int_{ B(R_i)} u^{p_i} d\mu \right)^\theta,
\]

This yields
\begin{align*}
 \left( \int_{B(R_{i+1})} u^{p_{i+1}} d\mu\right)^{\frac{1}{p_{i+1}}}
  \leq \left( \frac{C_2 e^{2C_e K R^2} V^{1-\theta}}{ (1-\delta)^{2}} \right)^{\sum_{j=1}^{i+1} \frac{1}{p\theta^{j}}} 2^{2\sum_{j=1}^{i+1} \frac{j}{p\theta^{j}}}
   \left( \int_{B(R)} u^{p} d\mu\right)^{\frac{1}{p}} .
\end{align*}
Simple computation shows that
\[
 \sum_{j=1}^\infty \frac{1}{\theta^j} = \frac{1}{\theta-1} = \frac{Q}{2}, \quad \sum_{j=1}^\infty \frac{j}{\theta^j}
 = \frac{\theta}{(\theta-1)^2} = \frac{Q(Q+2)}{4}, \quad \lim_{i\rightarrow \infty} R_i = \delta R
\]
\[
 \lim_{i\rightarrow \infty} \left( \int_{B(R_{i+1})} u^{p_{i+1}} d\mu\right)^{\frac{1}{p_{i+1}}} = \sup_{ \delta B} \{ u \},
\]
Conclusively, where $C_3= C_2^{\frac{Q}{2}} 2^{\frac{Q(Q+2)}{2}} $, we have
\begin{align*}
 \sup_{ \delta B} \{ u \} \leq \left(  C_3 e^{Q C_e K R^2} (1-\delta)^{-Q} V^{-1}  \right)^{\frac{1}{p}} \left( \int_{B(R)} u^{p} d\mu\right)^{\frac{1}{p}}.
\end{align*}
\end{proof}

\begin{corollary}[$L^p$ mean value inequality, $0<p< 2$] \label{thm:Lp mean value p<2}
$L^p$ mean value inequality (\ref{ineq:Lp mean value p>2}) also holds for any $0<p<2$ with the constant $C_3$ replaced by some $C_4=C(Q,p)$.
In particular, for $p=1$
\begin{align*}
 \sup_{\delta B}\{u\} \leq C_m e^{c_m KR^2} (1-\delta)^{-Q} \left( \frac{1}{\mu(B)}\int_B u d\mu \right),
\end{align*}
where $C_m,c_m>0$ depend only on $\rho_2,\kappa,d$.
\end{corollary}

\begin{proof}
Let $\frac{1}{2} < \sigma < 1$ and $\rho = \sigma+(1-\sigma)/4$. \\
By Theorem \ref{thm:Lp mean value p>2} (we can pick $R_0=\rho R$ in the proof) and
$ \int_B u^2 d\mu \leq \sup_{B} \{u^{2-p} \} \int_B u^p d\mu $,
\begin{align*}
\sup_{\sigma B} \{ u \} & \leq C e^{ c K R^2}(1-\sigma)^{-\frac{Q}{2}}\left( V^{-1} \int_{\rho B} u^2 d\mu \right)^{\frac{1}{2}} \\
          &\leq \left( C V^{-\frac{1}{2}} (\int_B u^p d\mu)^{\frac{1}{2}} e^{ c K R^2}\right)(1-\sigma)^{-\frac{Q}{2}} \sup_{\rho B} \{ u^{1-\frac{p}{2}} \} .
\end{align*}

\
Now fix $\delta \in (\frac{1}{2},1 )$ and set $\sigma_0 = \delta $, $\sigma_{i+1}=\sigma_{i}+(1-\sigma_{i})/4 =\sigma_i+ (\frac{3}{4})^i (1-\delta)/4 $. Then
\begin{align*}
\sup_{\sigma_i B}\{u\} \leq \Lambda (\frac{4}{3})^{Qi/2} (1-\delta)^{-\frac{Q}{2}} (\sup_{\sigma_{i+1} B} \{u\})^{1-\frac{p}{2}},
\end{align*}
where $\Lambda=\left( C (V^{-1}\int_B u^p d\mu)^{\frac{1}{2}} e^{ c K R^2}\right)$.

\
Finally, the same iteration of Theorem \ref{thm:Lp mean value p>2} yields
\begin{align*}
\sup_{\delta B}\{u\} &\leq (\frac{4}{3})^{\frac{2Q}{p^2}} \Lambda^{\frac{2}{p}}(1-\delta)^{-\frac{Q}{p}} \\
 & = (\frac{4}{3})^{\frac{2Q}{p^2}} C^{\frac{2}{p}} e^{\frac{2c}{p}K R^2} (1-\delta)^{-\frac{Q}{p}} \left( V^{-1} \int_B u^p d\mu \right)^{\frac{1}{p}}.
\end{align*}

\end{proof}

\section{Uniqueness of the positive solution }

\subsection{Minimality of the heat semigroup for positive solutions }

\
To prove Theorem \ref{thm:uniq of positive}, we reduce the question to the zero initial data. Following Lemma \ref{lem:minimality} enables the reduction.
This section is based on the idea of \cite{DON}.


\begin{lemma}[minimality of the heat semigroup] \label{lem:minimality}  Let $u \in C(\M\times(0,T))$ 
be a non-negative supersolution of the
heat equation (\ref{heat_eq}) with initial data $f \in L^2_{loc}(\M)$, $f\geq0$. \\
Then $P_t f(x)=\int_\M p_t(x,y) f(y)d\mu(y)$ is a smooth solution of (\ref{heat_eq})
satisfying $P_t f \xrightarrow{L^2_{loc}} f$ as $t\rightarrow 0$ and $u(\cdot,t)\geq P_t f$.
\end{lemma}

\begin{proof}
\
For any $\Omega \Subset \M$,
we denote $P_t^{\Omega,D} $ the Dirichlet heat semigroup associated with $L$ on $\Omega$.
Using the maximum principle for $P_t^{\Omega,D} f - u(\cdot,t)$, we have
\[
 u(x,t)\geq P_t^{\Omega,D} f(x), \quad \forall x\in \Omega
\]

\
Denote $f_k = f 1_{\Omega_k} \in L^2(\M)$ for the exhaustion $\{ \Omega_k\}$. As shown above, $u(\cdot,t)\geq P_t^{\Omega_k,D} f\geq P_t^{\Omega_k,D} f_i$ for all $i$.
Since $P^{\Omega_k,D}_t f_i \xrightarrow{L^2(\M)} P_t f_i$ as $k\rightarrow \infty$, we have $u(\cdot,t) \geq P_t f_i $ almost everywhere for all $i$.
Therefore,
\[
 u(\cdot,t) \geq P_t f \quad \mbox{ almost everywhere}.
\]

\
To prove that the smooth $P_t f$ solves the heat equation, first we see that $P_t f\in L^1_{loc}(\M) $ from the above estimate.
Denote
\[
 u_k = P_t (\min(f,k) 1_{\Omega_k} ),
\]
then $u_k$ is a smooth solution of the subelliptic heat equation and
$ u_k \nearrow P_t f $ as $k\rightarrow \infty$ at any $(x,t)\in \M\times(0,T) $.

\
For any $\varphi \in C_0^\infty(\M\times(0,T))$, since $P_t f\in L^1_{loc}(\M) $,
\[
 \left| (\partial_t \varphi +L\varphi) u_k\right| \leq (\sup |\partial_t \varphi + L\varphi |) 1_{\supp \varphi} P_t f \in L^1(\M),\quad \forall k\in \mathbb N.
\]
This allows us to take the limit of the integrand on the left hand side of
\begin{align*}
  \int_0^T \int_\M (\partial_t \varphi +L\varphi) u_k d\mu dt = \int_0^T \int_\M \varphi \left(L-\partial_t\right)u_k d\mu dt =0.
\end{align*}
Therefore $P_t f$ is a distributional solution of
the subelliptic heat equation, and also it is smooth by the smooth convergence of $u_k$ to $P_t f$ and the hypoellipticity of $L-\partial_t$.

\
Once the smoothness of $P_t f$ and $u\geq P_t f$ are proved, the initial condition is straightforward as follows :
On any $\Omega \Subset \M$,
\[
 P_t (f 1_\Omega) \leq P_t f \leq u(\cdot,t).
\]
When $t\rightarrow 0$, $u\xrightarrow{L^2(\Omega)} f $ and $ P_t (f 1_\Omega)\xrightarrow{L^2(\M)} f 1_\Omega $. Hence $P_t f \xrightarrow{L^2(\Omega)} f$.
\end{proof}

\subsection{proof of Theorem \ref{thm:uniq of positive} }
\
From the minimality Lemma \ref{lem:minimality}, for any non-negative continuous solution $u$ of (\ref{heat_eq}),
\[
 w(x,t) = u(x,t)-P_t u(x,0)
\]
is a non-negative solution of (\ref{heat_eq}) with zero initial data.
Thus we can reduce the uniqueness of the positive solution to the zero initial data case.

\
Let $w(x,t)$ be any non-negative solution of the heat equation (\ref{heat_eq}) with initial data $f\equiv0$.\\
Define $v(x,t) = \int_0^t w(x,s)ds$. Our goal is to show $v \equiv 0$, and so is $w$.

\begin{remark}
$v(x,t) = \int_0^t w(x,s)ds$ is a non-negative solution of the heat equation (\ref{heat_eq}) with zero initial data, and
subharmonic in $x$, i.e. $L v(\cdot,t) = \int_0^t L w(\cdot,s)ds = w(\cdot,t) \geq 0 $.
\end{remark}

\
The following growth estimate condition is originally suggested by Tikhonov for the uniqueness of the solution for the heat equation.

\begin{proposition} [Growth estimate of the solution, Tikhonov's condition]
For any $\epsilon >0$ and $0\leq t \le \epsilon$, if $v \in C(\M \times (0,\epsilon))$
is a non-negative solution of the subelliptic heat equation (\ref{heat_eq}) satisfying $Lv(\cdot,t) \geq 0$, then
\[
 v(x,t) \le C_1 \exp (C_2 d^2(p,x)),
\]
where $C_1 = C_1(\epsilon)>0$, $C_2 = C_2 (\epsilon)>0$, and $d(p,\cdot)$ is the distance from a fixed $p\in \M$.
\end{proposition}

\begin{proof}
Let $B=B(x, d(p,x)+1 )$. Fix $T>0$. From the minimality Lemma \ref{lem:minimality},
\[
 v(p,t+T) \geq P_T v(\cdot,t) = \int_M p_T(p,y) v(y,t) d\mu(y) \geq \int_B p_T(p,y) v(y,t) d\mu(y).
\]

\
From the curvature condition $CD(-K,\rho_2,\kappa,d)$, the lower bound of heat kernel (\ref{ineq:hk LB}) is
\begin{align*}
 p_T(p,y)  &\geq C_3 \exp(-C_4 d^2(p,y) ),
\end{align*}
where $C_3 = C_3(p,T,K,\rho_2,\kappa,d)>0$, $C_4 = C_4(T,K,\rho_2,\kappa,d)>0$.

\
By the triangle inequality $d(p,y) \leq 2d(p,x) +1$ for $y\in B$,
\[
 \int_B v(y,t)d\mu(y) \leq C_5 \exp (C_6 d^2(p,x)) ~ v(p,t+T).
\]
By $L^1$ mean value estimate of Corollary \ref{thm:Lp mean value p<2} for the subharmonic function $v(\cdot,t)$
\[
 v(x,t) \leq C_7 \exp(C_8 d(p,x)^2) \int_B v(y,t) d\mu(y),
\]
where $C_7, C_8>0 $ depend on $ K,\rho_2,\kappa,d$.
\
Therefore, we obtain
\[
 v(x,t) \leq C_9 \exp(C_{10} d^2 (p,x)) v(p,t+T),
\]
where the constants depend on $p,T,K,\rho_2,\kappa,d $.
As $t$ varies from $0$ to $\epsilon$,  $v(p,t+T)$ remains uniformly bounded in $t$. So we have the desired conclusion.
\end{proof}

\
Together with the previous proposition, the proof of Theorem \ref{thm:uniq of positive} is finished by the following proposition.

\begin{proposition}
 If $v(x,t)$ is a solution of (\ref{heat_eq}) with initial $f(x)\equiv 0$ satisfying
\[
 |v(x,t)| \le C_1 \exp{C_2 d^2(p,x)}
\]
for some positive $C_1$, $C_2$, then $v\equiv 0$.
\end{proposition}
Existence of Lipschitz cut-off function and integration by part allow us to follow exactly the same proof of corollary 11.10 in \cite{Gryg}.

\section{Uniqueness of $L^p$ solution}

\subsection{proof of Theorem \ref{thm:uniq of Lp}, $p>1$}

For $p=\infty$, the uniqueness of the $L^\infty$ solution, or equivalently the stochastic completeness of $\M$, can be found in \cite{BG}.
If $p\in(1,\infty)$, without any curvature assumption, the uniqueness follows immediately by adapting the idea of \cite{Li}.
\begin{theorem}
Let $v(x,t)$ is a non-negative function defined on $\M \times (0,T)$ with
\begin{align*}
 & \left( \frac{\partial}{\partial t} -L \right) v(x,t)  \leq 0    \\
 & v\xrightarrow{L^p_{loc}} 0 \textmd{  as } t\rightarrow 0  \\
 & v(\cdot,t) \in L^p(\M) \quad \forall t\in (0,T),
\end{align*}
then $v(x,t)\equiv 0$ on $\M\times (0,T)$.\\
In particular, any $L^p$ solution of the heat equation is uniquely determined by its initial data in $L^p(\M)$.
\end{theorem}

\begin{proof}
Fix $x_0 \in \M$ an arbitrary base point.

From remark \ref{remark:cutoff}, we choose $\psi(x) \in C_0 (B(x_0,2R))$ a cut-off function satisfying $\psi|_{B(x_0,R)} \equiv 1 $, $0\leq \psi \leq 1$,
$\| \sqrt{\Gamma(\psi)}\|_\infty \leq \frac{C}{R}$ for some $C>0$.

\
Since $v$ is a subsolution with the zero initial data, for any $\tau\in (0,T)$,
\begin{align*}
\int_0^\tau \int_\M & \psi^2(x) v^{p-1}(x,t) L v(x,t) d\mu(x) dt   \geq \int_0^\tau \int_\M \psi^2(x) v^{p-1} \frac{\partial v}{\partial t} d\mu(x) dt \\
 & = \frac{1}{p} \int_0^\tau \frac{\partial }{\partial t} \left( \int_\M \psi^2(x) v^{p} d\mu(x)\right) dt
   = \frac{1}{p}\int_\M \psi^2(x) v^{p}(x,\tau) d\mu(x).
\end{align*}
On the other hand, integrating by parts yields
\begin{align*}
\int_0^\tau \int_\M & \psi^2(x) v^{p-1}(x,t) L v(x,t) d\mu(x) dt \\
 &= - \int_0^\tau \int_\M 2\psi v^{p-1} \Gamma(\psi,v)  d\mu dt - \int_0^\tau \int_\M  \psi^2 (p-1) v^{p-2} \Gamma(v)  d\mu dt .
\end{align*}
On the right hand side, observing
\begin{align*}
0\leq & \left(\sqrt{\frac{2}{p-1}\Gamma(\psi)}v - \sqrt{\frac{p-1}{2}\Gamma(v)}\psi \right)^2  \leq \frac{2}{p-1}\Gamma(\psi)v^2 + 2 \Gamma(\psi,v) \psi v +\frac{p-1}{2}\Gamma(v)\psi^2 ,
\end{align*}
we obtain the following estimate.
\begin{align*}
 \int_0^\tau \int_\M & \psi^2(x) v^{p-1}(x,t) L v(x,t) d\mu(x) dt \\
 & \leq  \int_0^\tau \int_\M  \frac{2}{p-1} \Gamma(\psi) v^p  d\mu dt - \int_0^\tau \int_\M  \frac{p-1}{2}\psi^2 v^{p-2} \Gamma(v)  d\mu dt \\
 &= \int_0^\tau \int_\M  \frac{2}{p-1} \Gamma(\psi) v^p  d\mu dt - \frac{2(p-1)}{p^2} \int_0^\tau \int_\M  \psi^2 \Gamma(v^{p/2})  d\mu dt .
\end{align*}
Combining with the previous conclusion and the assumption $|\sqrt{\Gamma(\psi)}|\leq \frac{C}{R}$,
\[
 \int_\M \psi^2(x) v^{p}(x,\tau) d\mu(x) + \frac{2(p-1)}{p} \int_0^\tau \int_\M  \psi^2 \Gamma(v^{p/2})  d\mu dt
   \leq \frac{2 pC^2}{(p-1) R^2} \int_0^\tau \int_\M   v^p  d\mu dt.
\]
As $R \rightarrow \infty$, since $ \Gamma(v^{p/2})\geq 0 $, we have
\[
 \int_\M v^p(x,\tau) d\mu(x) =0 \quad \forall \tau \in (0,T) .
\]
Thus, $v\equiv 0$.
\end{proof}

\subsection{Hamilton's inequality}

\
Before we move on to $L^1$ solutions, we will prove the gradient estimate of the logarithm of the heat kernel. 
We will apply subelliptic version of Hamilton's inequality which was originally proved for closed Riemannian manifolds in \cite{Ham}, 
then for non-compact Riemannian manifolds in \cite{Kot}.

\begin{proposition} [Hamilton's inequality]
Assume that $\M$ satisfies the curvature condition $CD(-K,\rho_2,\kappa,d)$. If a positive solution $u\in \mathcal A_\epsilon$ to the subelliptic heat 
equation satisfies $u \leq M$ on $\M\times (0,T)$ for some $M>0$ and $0<T\leq\infty$, one has
\begin{align}  \label{ineq:Hamiltons inequality}
 t \Gamma( \ln u(x,t) ) \leq \left( 1+\frac{2\kappa}{\rho_2}+2Kt \right) \ln \left( \frac{M}{u(x,t)}\right)
\end{align}
for all $(x,t)\in \M \times (0,T)$.
\end{proposition}
\begin{proof}
By Theorem \ref{thm:uniq of positive}, it suffices to show that the estimate holds for $u=P_t f \in \mathcal A_\epsilon >0$. (See Remark \ref{rmk:A epsilon} for $\mathcal A_\epsilon$.)
We apply the reverse log-Sobolev inequality in \cite{BB}, i.e.,
\begin{align*}
 t P_t f \Gamma(\ln P_t f) + \rho_2 t^2 P_t f \Gamma^Z (\ln P_t f) \leq \left(1+\frac{2\kappa}{\rho_2}+2Kt \right) \left( P_t (f\ln f) - (P_t f) \ln P_t f\right).
\end{align*}
Then our desired inequality is instantly deduced by $P_t (f\ln f) \leq (P_t f) \ln M $ and $\rho_2 t^2 P_t f \Gamma^Z \geq 0$.
\end{proof}

\begin{lemma} \label{lem:gradient of log heat kernel}
If $\M$ satisfies $CD(-K,\rho_2,\kappa,d)$, there exists $C_h=C_h(\rho_2,\kappa,d)>0$, $t>0, x,y\in\M$,
\begin{align*}
 \Gamma_x(\ln p_t(x,y)) \leq \frac{C_h}{t} \left( 1+\frac{2\kappa}{\rho_2}+Kt \right)  \left( K(t+ d(x,y)^2)+ \frac{d(x,y)^2}{t}\right).
\end{align*}
\end{lemma}

\begin{proof}
Let $t>0$ and $y\in \M$. Let $u(x,s) := p_{\frac{t}{2} +s} (x,y)$, then $u$ is a smooth, positive solution to the heat equation.
From the heat kernel upper bound (\ref{ineq:hk UB2}), for $t_0= \frac{1}{6C_6 K}$, $0<t<t_0$, $0\leq s\leq \frac{t}{2}$, $\forall x\in \M $,
\begin{align*}
 u(x,s) & \leq \frac{C_{5}}{\mu(B(y,\sqrt{\frac{t}{2}+s})) } \exp\left( C_{6} K(\frac{t}{2}+s +d(x,y)^2)-\frac{d(x,y)^2}{6(\frac{t}{2}+s)} \right) \\
    & \leq \frac{C_6 e^{1/6}}{\mu(B(y,\sqrt{\frac{t}{2}+s}))} \leq \frac{C_6'}{\mu(B(y,\sqrt{\frac{t}{2}}))} =M.
\end{align*}
Moreover $u(x,s)\leq M$ for all $s>0$, since $\|P_t\|_{\infty\rightarrow \infty} \leq 1$ for any $t>0$.

\
By the Hamilton's inequality (\ref{ineq:Hamiltons inequality}), the heat kernel lower bound (\ref{ineq:hk LB}) for $u(x,s)$ with $s=\frac{t}{2}$ and
the doubling property (\ref{ineq:doubling2}),
\begin{align*}
 &\frac{t}{2} \Gamma( \ln u(x,\frac{t}{2}) ) \leq \left( 1+\frac{2\kappa}{\rho_2}+Kt \right) \ln \left( \frac{M}{ u(x,\frac{t}{2}) } \right) \\
 &\leq \left( 1+\frac{2\kappa}{\rho_2}+Kt \right) \frac{C_h}{2} \left( K(t+ d(x,y)^2)+ \frac{d(x,y)^2}{t} \right) ,
\end{align*}
where $C_h = 2 \ln\left(C_6' C_1^{-1}C_{d1} 2^{Q/2}\right)\max(\frac{D}{2d}, \frac{4C_{d2}}{3}+C_2)$.
\end{proof}

\
If we combine the previous lemma with (\ref{ineq:hk UB2}), we obtain the following simpler statement for small $t$, which will be useful in the next section.
\begin{lemma} \label{lem:gradient of log heat kernel2}
Assume $CD(-K,\rho_2,\kappa,d)$. For any $R>0$, $\beta>0$ and $x_0\in \M$, there exists $C>0$, $t_0>0$ such that for $d(x,y)\geq R/4$, $0<t<t_0$,
\begin{align*}
 \sqrt{\Gamma_y(p_t(x,y) )} \leq \frac{ C e^{-\beta R^2} }{\mu(B(x,\sqrt t ))} .
\end{align*}
\end{lemma}

\subsection{proof of Theorem \ref{thm:uniq of Lp}, $p=1$}

\
Prior to the uniqueness of $L^1$ solution for the heat equation, we prove the uniqueness of $L^1$ harmonic function.
Basic idea of the proof comes from \cite{Li}.

\begin{remark}
We assume the fixed curvature bound $\rho_1=-K$ instead of the negative quadratic lower bound of Ricci curvature of \cite{Li}.
\end{remark}

\begin{theorem} \label{thm:uniq of L1 harmonic}
If $\M$ satisfies $CD(-K,\rho_2,\kappa,d)$, then
any $L^1$ non-negative subharmonic function on $\M$ must be identically constant. \\
In particular, any $L^1$ harmonic function on $\M$ must be identically constant.
\end{theorem}

\begin{proof}
Let $g\in L^1(\M)$ be a non-negative function satisfying $Lg\geq 0$, i.e. subharmonic. For any $t>0$,
\begin{align*}
L P_t g (x) =& \int_\M  (L_x p_t(x,y)) g(y) d\mu(y) \\
 = \int_\M & ( \frac{\partial}{\partial t} p_t(x,y) ) g(y) d\mu(y) = \int_\M ( L_y p_t(x,y) ) g(y) d\mu(y).
\end{align*}

\
We claim the following integration by parts.
\begin{align} \label{claim:IBP}
\int_\M ( L_y p_t(x,y) ) g(y) d\mu(y) = \int_\M p_t(x,y) L g(y) d\mu(y).
\end{align}

\
To justify the claim, we observe the following
\begin{align}
\nonumber & \left| \int_\M  \psi_R(y) \big[ g(y) L_y p_t(x,y) - p_t(x,y) L g(y) \big] d\mu(y) \right| \\
\nonumber & = \left| \int_\M -\big[ \Gamma( \psi_R g , p_t(x,\cdot)) - \Gamma( \psi_R p_t(x,\cdot ) , g ) \big] d\mu \right| \\
\nonumber & = \left| \int_\M -\big[ g \Gamma( \psi_R , p_t(x,\cdot)) - p_t(x,\cdot ) \Gamma( \psi_R  , g ) \big]d\mu \right| \\
 & \label{claim:IBP2} \leq \int_{B(x_0,R+1)\setminus B(x_0,R)}  C \left( g \sqrt{\Gamma( p_t(x,\cdot))} + p_t(x,\cdot ) \sqrt{\Gamma( g )} \right)d\mu ,
\end{align}
where $\psi_R \geq 0 $ is a Lipschitz continuous cut-off function
satisfying $\psi_R |_{B(x_0,R)}=1$, $ \supp \psi_R \subset B(x_0,R+1)$
and $\sqrt{\Gamma(\psi_R)} \leq C$ almost everywhere for some $C>0$ which is independent to $R>0$.(See Remark \ref{remark:cutoff}.)

\
It suffices to show that both integrals on the right-hand side vanish as $R \rightarrow \infty$. We can consider
$R$ large enough so that $x\in B(x_0,R/4)$.


\
Let $\varphi$ be a cut-off function for an annulus satisfying $\varphi|_{B(x_0,R+1)\setminus B(x_0,R)} = 1$, $\varphi|_{B(x_0,R-1)\cup (\M\setminus B(x_0,R+2))} = 0$
and $\sqrt{\Gamma(\varphi)} \leq C$ almost everywhere. By the subharmonicity of $g$,
\begin{align*}
0 & \leq \int_\M \varphi^2 g Lg d\mu = -2 \int_\M \varphi g \Gamma(\varphi,g) d\mu - \int_\M \varphi^2 \Gamma(g) d\mu \\
& \leq \int_\M \left[ -\frac{1}{2}( 2g\sqrt{\Gamma(\varphi)} -\varphi\sqrt{\Gamma(g)} )^2 + 2 \Gamma(\varphi) g^2 - \frac{1}{2} \varphi^2 \Gamma (g) \right]d\mu \\
&\leq  2\int_\M \Gamma(\varphi) g^2 d\mu - \frac{1}{2} \int_\M \varphi^2 \Gamma (g) d\mu.
\end{align*}

\
Therefore, applying $L^1$ mean value estimate Corollary \ref{thm:Lp mean value p<2} to $g$,
\begin{align*}
 &\int_{B(x_0,R+1)\setminus B(x_0,R)} \Gamma(g) d\mu \leq 4 \int_\M \Gamma(\varphi) g^2 d\mu \leq 4C^2\int_{B(x_0,R+2)} g^2 d\mu \\
 & \leq 4C^2  \|g \|_{L^1} \sup_{B(x_0,R+2)} g(y) \leq C' e^{c' K R^2} \frac{1}{\mu(B(x_0,2R+4))} \| g \|_{L^1}^2 .
\end{align*}

\
By Schwarz inequality,
\begin{align*}
\int_{B(x_0,R+1)\setminus B(x_0,R)}& \sqrt{\Gamma(g)} d\mu \leq \left( \int_{B(x_0,R+1)\setminus B(x_0,R)} \Gamma(g) d\mu \right)^{\frac{1}{2}} (\mu(B(x_0,2R+4)))^{\frac{1}{2}}\\
& =C e^{c K R^2}\| g \|_{L^1}.
\end{align*}

\
In addition to this estimate, to bound the second integration in (\ref{claim:IBP2}) we consider the upper bound of heat kernel (\ref{ineq:hk UB2}):
\begin{align*}
 p_t(x,y) & \leq \frac{C_{5}}{\mu(B(x,\sqrt t)) } \exp\left( C_{6} K(t+d(x,y)^2)-\frac{d(x,y)^2}{6t} \right).
\end{align*}

\
Combining the above two inequalities, we estimate the second term of (\ref{claim:IBP2}) for small $0<t<T=T(K,\rho_2,\kappa,d)$.
\begin{align*}
 &\int_{B(x_0,R+1)\setminus B(x_0,R)} p_t(x,\cdot) \sqrt{\Gamma( g )} d\mu \\
 &\quad \leq \left(\sup_{y\in B(x_0,R+1)\setminus B(x_0,R)} p_t(x,y) \right) \int_{B(x_0,R+1)\setminus B(x_0,R)} \sqrt{\Gamma( g )} d\mu \\
 &\quad \leq \frac{C }{\mu(B(x,\sqrt t)) }  \exp\left(-\alpha \frac{R^2}{t}\right) \| g \|_{L^1} \xrightarrow{R\rightarrow \infty} 0,
\end{align*}
with $d(x_0,x)\leq \frac{R}{4} $, $\frac{R}{2} \leq d(x,y) \leq 4R $ and $\alpha>0$.\\


\
For the first term of (\ref{claim:IBP2}), $L^1$ mean value estimate for $g$ yields
\begin{align*}
 &\int_{B(x_0,R+1)\setminus B(x_0,R)} g \sqrt{\Gamma(p_t(x,\cdot))} d\mu \\
 &\quad \leq \left(\sup_{B(x_0,R+1)\setminus B(x_0,R)} g \right) \int_{B(x_0,R+1)\setminus B(x_0,R)} \sqrt{\Gamma(p_t(x,\cdot))} d\mu\\
 &\quad \leq \left( \frac{C e^{cKR^2}}{ \mu(B(x_0,2R+2))}\|g\|_{L^1} \right) \int_{B(x_0,R+1)\setminus B(x_0,R)} \sqrt{\Gamma(p_t(x,\cdot))} d\mu.
\end{align*}
If we apply Lemma \ref{lem:gradient of log heat kernel2} for $\beta >cK$,
\begin{align*}
 &\quad \leq \left( \frac{\mu(B(x_0,R+1)) }{ \mu(B(x_0,2R+2))}\|g\|_{L^1} \right) \frac{C e^{-\beta' R^2}}{ \mu(B(x,\sqrt t ))}
  \leq \|g\|_{L^1} \frac{C e^{-\beta' R^2}}{ \mu(B(x,\sqrt t ))}  \xrightarrow{R\rightarrow \infty} 0,
\end{align*}
where $0<t<T=T(K,\rho_2,\kappa,d)$ small enough and $\beta'>0$.\\

\
Therefore, as $R\rightarrow \infty$, the integration of (\ref{claim:IBP2}) vanishes as we desired, and we proved our claim (\ref{claim:IBP}).\\

\
Now since the integration by part (\ref{claim:IBP}) holds for small $t$, we have
\begin{align*}
 \frac{\partial}{\partial t} P_t g = L P_t g = P_t (Lg) \geq 0.
\end{align*}
And by the semigroup property, $P_t g(x)\geq g(x)$ for all $t>0$, $x\in \M$.

\
On the other hand, by the stochastic completeness (\cite{BG}) of $P_t$, $\| P_t g \|_{L^1} = \| g \|_{L^1} $. Therefore $P_t g = g$, i.e. $g$ is harmonic.\\

\
For any constant $\gamma>0$, $(g-\gamma)_+ = \max(0,g-\gamma)\leq g$ is also a non-negative $L^1$ subharmonic function. And by the same argument, it is harmonic.
$\min(g,\gamma) = g - (g-\gamma)_+ $ is also non-negative $L^1$ harmonic function. Observe that $\min(g,\gamma) \in C^\infty(\M)$ for any $\gamma>0$ by the hypoellipticity of $L$.
This is not possible unless $g$ is constant.

\
Finally, any harmonic function $u\in L^1(\M)$ is identically constant since $|u|$ is non-negative $L^1$ subharmonic function which must be constant by the above.
\end{proof}

\
With the uniqueness of $L^1$ harmonic function, we are ready to prove $L^1$ uniqueness of the solution for the subelliptic heat equation.
\begin{theorem}
Let $\M$ satisfy $CD(-K,\rho_2,\kappa,d)$. Let $v : \M \times [0,\infty) \rightarrow \R $ be a non-negative function satisfying
\begin{align*}
 &\left( L-\frac{\partial}{\partial t} \right) v(x,t) \geq 0, \quad \quad \| v(\cdot,t)\|_{L^1(\M)} < \infty, \quad \forall t>0, \\
 & \| v(\cdot,t) \|_{L^1(\M)} \xrightarrow{t\rightarrow 0} 0,
\end{align*}
then $v(x,t) \equiv 0 $ on $\M \times (0,\infty)$.
\end{theorem}

\begin{proof}
For any $\epsilon>0$, denote
\begin{align*}
 \psi_\epsilon(x,t) = \max\big( 0, v(x,t+\epsilon)-P_t ( v(\cdot,\epsilon)) \big) .
\end{align*}
Then it follows that $\psi_\epsilon\geq0$, $\lim_{t\rightarrow0} \psi_\epsilon(x,t)=0$, $\left( L-\frac{\partial}{\partial t}\right)\psi_\epsilon\geq 0$.

\
Fix $T>0$. Define
\[
f(x) = \int_0^T \psi_\epsilon (x,t) dt,
\]
which satisfies $ Lf(x) = \psi_\epsilon(x,T)-\psi_\epsilon(x,0)\geq0$.

\
The assumption $v(\cdot,t)\in L^1(\M)$ yields $\int_0^T \int_\M |v(x,t+\epsilon)| d\mu(x) dt  <\infty$. Together with $\int_0^T \int_\M P_t v(x ,\epsilon) d\mu dt \leq T\int_\M v(x,\epsilon)d\mu(x) <\infty $, we obtain $\|f\|_{L^1(\M)} <\infty $.

\
Now $f$ is non-negative $L^1$ subharmonic function, so that we can apply Theorem \ref{thm:uniq of L1 harmonic} to $f$ and conclude $f$ is identically constant.
This implies $0=Lf(\cdot)=\psi_\epsilon(\cdot,T)$ for arbitrary $T>0$. Hence for any $t>0$,
\begin{align*}
 & v(x,t+\epsilon) \leq P_t (v(\cdot,\epsilon ))(x) \\
 & \quad \leq \| p_t(x,\cdot)\|_\infty \| v(\cdot,\epsilon) \|_{L^1} \leq M \| v(\cdot,\epsilon) \|_{L^1} \xrightarrow{\epsilon\rightarrow 0} 0,
\end{align*}
where the uniform bound for $\| p_t(x,\cdot)\|_\infty$ is found in Lemma \ref{lem:gradient of log heat kernel}.

\
Therefore non-negative $v(x,t)$ must be zero for all $(x,t)\in \M \times (0,\infty)$.
\end{proof}

\begin{proof}[Proof of Theorem \ref{thm:uniq of Lp}, $p=1$]

\
For any $L^1$ solution $u$ of $\left(L-\frac{\partial}{\partial t}\right)u=0 $ with the initial condition $u\xrightarrow{L^1} f \in L^1(\M)$ as $t\rightarrow 0$,
\begin{align*}
 v(x,t) := \left| u(x,t) - P_t f (x) \right|
\end{align*}
will be a non-negative $L^1$ subsolution of the heat equation with $v\xrightarrow{L^1} 0$ as $t\rightarrow 0$.

\
By the previous theorem, $v\equiv 0$ on $\M \times (0,\infty) $. Therefore, $u$ is uniquely determined to be $P_t f$.

\end{proof}

\begin{remark}
In \cite{AgLe},\cite{AgLe2}, the measure contractive definition of Ricci tensor bound and volume comparison theorem (which is not yet established in our framework)
are introduced in three dimensional sub-Riemannian spaces. This measure contraction property is extended to higher dimensions in \cite{LeeLZ}.

\
One can find the uniqueness theorem of the positive solution in symmetric local Dirichlet spaces, provided a local parabolic Harnack inequality in \cite{ElSc}.
\end{remark}


\begin{acknowledgements}
The author is grateful to his advisor, Fabrice Baudoin, for suggesting the problem and for his advice, numerous guidances throughout this work.
\end{acknowledgements}

\end{document}